\documentclass[10pt]{amsart}
\usepackage{amsmath,amssymb,latexsym,cancel,rotating}
\usepackage{tikz-cd}
\usepackage{tikz}
\usepackage{graphicx,mathrsfs,color,fancyhdr,amsthm, float}
\usepackage[all]{xy}
\usepackage{geometry}
\usepackage{hyperref}
\usepackage{setspace}
\newtheorem{theorem}{Theorem}[section]
\newtheorem{lemma}[theorem]{Lemma}
\newtheorem{corollary}[theorem]{Corollary}
\newtheorem{definition}[theorem]{Definition}
\newtheorem{proposition}[theorem]{Proposition}

\newtheorem{thm}{Theorem}

 \def\cD{{\mathcal D}}    
\def\cH{{\mathcal H}}

  \def\bbZ{{\mathbb Z}}  \def\bbQ{{\mathbb Q}}
    \def\bbF{{\mathbb F}}
\def\bbC{{\mathbb C}}

         \def\bfU{{\bf U}}
   
  \def\leq{\leqslant}

\def\Hom{\mbox{\rm Hom}}

\def\dim{\mbox{\rm dim}}   \def\End{\mbox{\rm End}}
\def\Aut{\mbox{\rm Aut}}

\def\mod{\mbox{\rm mod}}  \def\tr{\mbox{\rm tr}}
  
\def\rep{\mbox{\rm Rep}\,}

\def\bfV{{\mathbf V}}

\def\bfH{{\mathbf H}} \def\bfP{{\mathbf P}}

\def\supp{{\rm supp}}

\def\rep{{\rm rep}}

\def\bfV{\mathbf{V}}
\def\bfW{\mathbf{W}}
\def\bfE{\mathbf{E}}
\def\bfG{\mathbf{G}}

\def\bfU{\mathbf{U}}

\def\Ind{\mathbf{Ind}}

\newcommand\bbz{{\mathbb Z}}
\newcommand\bbq{{\mathbb Q}}

\newcommand\mk{{\mathcal{K}}}

\newcommand\mm{{\mathcal{M}}}

\newcommand\mf{{\mathcal{F}}}

\newcommand\mo{{\mathcal{O}}}
\newcommand\mP{{\mathcal{P}}}

\begin{document}

\title[]{Structure coefficients for quantum groups}

\author{Yixin Lan, Yumeng Wu, Jie Xiao}
\address{Max Plank institute for mathematics}
\email{lanyixin@amss.ac.cn (Y. Lan)}
\address{Beijing International Center for Mathematical Research, Beijing 100871, P. R. China}
\email{2506397175@pku.edu.cn (Y. Wu)}
\address{School of mathematical seciences, Beijing Normal University, Beijing 100875, P. R. China}
\email{jxiao@bnu.edu.cn(J.Xiao)}



\bibliographystyle{abbrv}
\begin{abstract}
According to the Hall algebras of quivers with automorphisms under Lusztig's construction, the polynominal forms of several structure coefficients for quantum groups of all finite types are presented in this note. We first provide a geometric realization of the coefficients between PBW basis and the canonical basis via standard sheaves on quiver moduli spaces with admissible automorphisms. This realization is constructed through Lusztig sheaves equipped with periodic functors and their modified Grothendieck groups. Second, within this geometric framework, we present an alternative proof for  the existence of Hall polynomials originally due to Ringel. Finally, we give a slight generalization of the Reineke-Caldero expression for the bar involution of PBW basis elements in symmetrizable cases. When the periodic functor $\mathbf{a}^*$ is taken $\operatorname{id}$, our results are the same as Lusztig's and Caldero-Reineke's.

\end{abstract}
\maketitle
\setcounter{tocdepth}{1}\tableofcontents
\begin{spacing}{1.5}
\section{Category with the periodic functor $\mathbf{a}^*$}
\subsection{Lusztig's Sheaves for Quivers with Automorphisms}

Most of the notations and properties in this section are taken from \cite[Chapters 11, 12, and 14]{lusztig2010introduction}. Lusztig considers a finite quiver $Q = (I, H, \Omega)$ with an admissible automorphism $a$ for a given Cartan matrix, in which a finite quiver $Q = (I, H, \Omega)$ consists of finite sets $I$, $H$, and $\Omega$, where $I$ is the set of vertices, $H$ is the set of all oriented arrows $s(h) \xrightarrow{h} t(h)$, and $\Omega \subset H$ is a subset such that $H$ is the disjoint union of $\Omega$ and $\bar{\Omega}$. Here, $\bar{}: H \to H$ is the map taking the opposite orientation. We call such $\Omega$ an orientation of $Q$.

For a finite quiver $(I, H, \Omega)$ without loops, an admissible automorphism $a$ consists of two permutations: $a: I \to I$ and $a: H \to H$, such that
\begin{itemize}
    \item[(a)] $a(s(h)) = s(a(h))$ and $a(t(h)) = t(a(h))$ for any $h \in H$;
    \item[(b)] $s(h)$ and $t(h) \in I$ belong to different $a$-orbits for any $h \in H$.
\end{itemize}

Given a symmetrizable generalized Cartan matrix $C = (c_{ij})_{i,j \in I'} = DB$, there is (not necessarily unique) a finite quiver $Q = (I, H, \Omega)$ with an admissible automorphism $a$ such that the $a$-orbits of $I$ are bijective to $I'$, the order of an $a$-orbit corresponding to $i \in I'$ equals $s_i$, and the number of arrows between $a$-orbits corresponding to $i, j \in I'$ equals $c_{ij}s_i = c_{ji}s_j$.

In this subsection, we introduce the category of Lusztig's sheaves for quivers. Let $\mathbf{k}$ be an algebraically closed field with $char(\mathbf{k}) = p > 0$, $q=p^r$, $r\in \mathbb{N}_{>0}$. Given a quiver $Q = (I, H, \Omega)$ and an $I$-graded $\mathbf{k}$-space $\mathbf{V} = \bigoplus_{i \in I} \mathbf{V}_i$ of dimension vector $\nu$, the affine variety $\mathbf{E}_{\mathbf{V}, \Omega}$ is defined by
$$
\bfE_{\nu, \Omega}=\mathbf{E}_{\mathbf{V}, \Omega} = \bigoplus_{h \in \Omega} \Hom(\mathbf{V}_{s(h)}, \mathbf{V}_{t(h)}).
$$

When the orientation is fixed, we could just denote it $\bfE_{\nu}$ or $\bfE_{\bfV}$.

The algebraic group $G_{\nu}=G_{\mathbf{V}} = \prod_{i \in I} GL(\mathbf{V}_i)$ acts on $\mathbf{E}_{\mathbf{V}, \Omega}$ by composition naturally. We denote the $G_{\mathbf{V}}$-equivariant derived category of constructible $\overline{\mathbb{Q}}_l$-sheaves on $\mathbf{E}_{\mathbf{V}, \Omega}$ by $\mathcal{D}^b_{G_{\mathbf{V}}}(\mathbf{E}_{\mathbf{V}, \Omega})$. We denote $[n]$ as the $n$-times shift functor, $\mathbf{D}$ as Verdier duality and $(n)$ as Tate twist.

Let $\mathcal{S}$ be the set of finite sequences $\boldsymbol{\nu} = (\nu^1, \nu^2, \dots, \nu^s)$ of dimension vectors such that each $\nu^l = a_l i_l$ for some $a_l \in \mathbb{N}_{\geqslant 1}$ and $i_l \in I$. If $\sum_{1 \leqslant l \leqslant s} \nu^l = \nu$, we say $\boldsymbol{\nu}$ is a flag type of $\nu$ or $\mathbf{V}$.

For a flag type $\boldsymbol{\nu}$ of $\mathbf{V}$, the flag variety $\mathcal{F}_{\boldsymbol{\nu}, \Omega}$ is the smooth variety consisting of pairs $(x, f)$, where $x \in \mathbf{E}_{\mathbf{V}, \Omega}$ and $f = (\mathbf{V} = \mathbf{V}^s \subseteq \mathbf{V}^{s-1} \subseteq \dots \subseteq \mathbf{V}^0 = 0)$ is a filtration of the $I$-graded space such that $x(\mathbf{V}^l) \subseteq \mathbf{V}^l$ and the dimension vector of $\mathbf{V}^{l-1} / \mathbf{V}^l = \nu^l$ for any $l$. There is a proper map $\pi_{\boldsymbol{\nu}}: \mathcal{F}_{\boldsymbol{\nu}, \Omega} \to \mathbf{E}_{\mathbf{V}, \Omega}$, given by $(x, f) \mapsto x$. Hence, by the decomposition theorem in \cite{BBD}, the complex
$$
L_{\boldsymbol{\nu}} = (\pi_{\boldsymbol{\nu}, \Omega})_! \overline{\mathbb{Q}}_l[\dim \mathcal{F}_{\boldsymbol{\nu}, \Omega}](\frac{\dim \mathcal{F}_{\boldsymbol{\nu}, \Omega}}{2})
$$
is a semisimple complex on $\mathbf{E}_{\mathbf{V}, \Omega}$, where $\overline{\mathbb{Q}}_l$ is the constant sheaf on $\mathcal{F}_{\boldsymbol{\nu}, \Omega}$.

We will denote the $i-th$ cohomology of perverse $t-$structure on $F\in \cD_c^b(X,\overline{\bbq_l})$ as $ ^p\mathbf{H}^i(F).$
\begin{definition}
    Let $\mathcal{P}_{\mathbf{V}}$ be the set consisting of those simple perverse sheaves $L$ in $\mathcal{D}^b_{G_{\mathbf{V}}}(\mathbf{E}_{\mathbf{V}, \Omega})$ such that $L$ is a direct summand (up to shifts) of $L_{\boldsymbol{\nu}}$ for some flag type $\boldsymbol{\nu}$ of $\mathbf{V}$. Let $\mathcal{M}_{\mathbf{V}}$ be the full subcategory of $\mathcal{D}^b_{G_{\mathbf{V}}}(\mathbf{E}_{\mathbf{V}, \Omega})$, consisting of direct sums of shifted simple perverse sheaves in $\mathcal{P}_{\mathbf{V}}$. Following \cite{OSNotes}, we call $\mathcal{M}_{\mathbf{V}}$ the category of Lusztig's sheaves for $Q$.
\end{definition}

The admissible automorphism gives a natural isomorphism $a: \mathbf{E}_{\mathbf{V}, \Omega} \to \mathbf{E}_{a(\mathbf{V}), \Omega}$ and induces the functor
$$
a^*: \mathcal{D}^b_{G_{\mathbf{V}}}(\mathbf{E}_{a(\mathbf{V}), \Omega}) \to \mathcal{D}^b_{G_{\mathbf{V}}}(\mathbf{E}_{\mathbf{V}, \Omega}).
$$
Later, we will denote $a^*$ as $\mathbf{a}^*$ for easy of reading. 
\begin{definition}
    Let $X$ be an algebraic variety over $\mathbf{k}$ with an $\bbF_q$-structure, $q=p^r,r\in \bbZ_{>0}$ and the corresponding Frobenius map $Fr^r:X\rightarrow X$, where $Fr$ induce $\bbF_p-$structure on $X$. Let $G$ be a connected algebraic group over $\mathbf{k}$ with an $\bbF_q$-structure and the Frobenius map $Fr^r:G\rightarrow G$ such that $G$ acts on $X$. We denote by $\cD_G^b(X)$ the $G$-equivariant bounded derived category of constructible $\overline{\bbQ}_l$-sheaves on $X$, $\cD^b_{G,m,r}(X)$ the subcategory consisting of mixed Weil complexes with Frobenius map $Fr^r:X\rightarrow X$.
\end{definition}
 We will denote ${Fr^r}^*$ as $\mathbf{Fr^r}^*$ for easy of reading.
For any mixed Weil complex $A\in \cD^b_{G,m}(X)$ with the Weil structure $\xi:\mathbf{Fr^r}^*(A)\rightarrow A$ and $x\in X^{Fr^r}, s\in \bbZ$, there is an isomorphism $H^s(\xi)_x:H^s(A)_x\rightarrow H^s(A)_x$ of the stalk at $x$ of the $s$-th cohomology sheaf. Taking the alternative sum of the traces of these isomorphisms, we obtain a value 
$$\chi_A(x)=\sum_{s\in \bbZ}(-1)^s\tr(H^s(\xi)_x)\in\overline{\bbQ}_l\cong \bbC$$
and a function $\chi_A\in \tilde{\cH}_{G^{Fr^r}}(X^{Fr^r})$. Moreover, $\chi_{-}$ induces a map from the Grothendieck group of $\cD^b_{G,m}(X)$ to $\tilde{\cH}_{G^{Fr^r}}(X^{Fr^r})$, see \cite[Lemma 5.3.12]{Pramod-2021}, which is called the trace map.
\subsection{Periodic Functors}

In this subsection, we review the definition of the periodic functor and refer to \cite[Chapter 11]{lusztig2010introduction} for more details. Let $o$ be a fixed positive integer.

Let $\mathcal{C}$ be a $\overline{\mathbb{Q}}_l$-linear additive category. A periodic functor on $\mathcal{C}$ is a linear functor $\mathbf{a}^*: \mathcal{C} \to \mathcal{C}$ such that $(\mathbf{a}^*)^o$ is the identity functor on $\mathcal{C}$.

\begin{definition}
Let $\mathbf{a}^*$ be a periodic functor on a category $\mathcal{C}$. We define the category $\tilde{\mathcal{C}}$ as follows:
\begin{itemize}
    \item[$\bullet$] Its objects are pairs $(A, \varphi)$, where $A \in \mathcal{C}$ and $\varphi: \mathbf{a}^*(A) \to A$ is an isomorphism in $\mathcal{C}$ such that the composition
    $$
    A = \mathbf{a}^{*o}(A) \xrightarrow{\mathbf{a}^{*(o-1)}(\varphi)} \mathbf{a}^{*(o-1)}(A) \to \dots \to \mathbf{a}^*(A) \xrightarrow{\varphi} A
    $$
    is the identity morphism on $A$.
    \item[$\bullet$] For any $(A, \varphi), (A', \varphi') \in \tilde{\mathcal{C}}$, the morphism space
    $$
    \Hom_{\tilde{\mathcal{C}}}((A, \varphi), (A', \varphi')) = \{ f \in \Hom_{\mathcal{C}}(A, A') \mid f\varphi = \varphi'(\mathbf{a}^*(f)) \}.
    $$
    \item[$\bullet$] The direct sum of $(A, \varphi), (A', \varphi') \in \tilde{\mathcal{C}}$ is defined naturally by $(A \oplus A', \varphi \oplus \varphi')$.
    
\end{itemize}
\end{definition}

\begin{definition}
An object $(A, \varphi) \in \tilde{\mathcal{C}}$ is called traceless if there exists an object $B \in \mathcal{C}$ and an integer $t \geqslant 2$ dividing $o$ such that $\mathbf{a}^{*t}(B) \cong B$, $A \cong B \oplus \mathbf{a}^*(B) \oplus \dots \oplus \mathbf{a}^{*(t-1)}(B)$, and $\varphi: \mathbf{a}^*(A) \to A$ corresponds to the isomorphism $\mathbf{a}^*(B) \oplus \mathbf{a}^{*2}(B) \oplus \dots \oplus \mathbf{a}^{*t}(B)$, taking $\mathbf{a}^{*s}(B)$ onto $\mathbf{a}^*(\mathbf{a}^{*s-1}(B))$ for $1 \leqslant s \leqslant t-1$ and taking $\mathbf{a}^{*t}(B)$ onto $B$, thus giving a permutation between the direct summands of $A$ and $\mathbf{a}^*A$.
\end{definition}

\begin{definition}
We say that $A$ and $B$ in $\tilde{\mathcal{C}}$ are isomorphic up to traceless elements if there exist traceless objects $C$ and $D$ such that $A \oplus C \cong B \oplus D$.
\end{definition}
\subsection{Canonical basis of symmetrizable cases}
As the graph below, 
$$\begin{tikzcd}
{\mf_{\boldsymbol{\nu},\Omega}} \arrow[r, "a"] \arrow[d, "\pi_{\boldsymbol{\nu}}"] & {\mf_{a\boldsymbol{\nu},\Omega}} \arrow[d, "\pi_{a\boldsymbol{\nu}}"] \\
\bfE_{\bfV} \arrow[r, "a"]                                                         & \bfE_{a\bfV}                                                         
\end{tikzcd}$$

There is a natural map $\phi_0$ given by the natural equivalence: $\phi_0:\mathbf{a}^*L_{a\boldsymbol{\nu}}\cong \mathbf{a}^*{\pi_{a\boldsymbol{\nu}}}_!(\overline{\bbQ_l}|_{\mf_{a\boldsymbol{\nu},\Omega}})\cong {\pi_{\boldsymbol{\nu}}}_!\mathbf{a}^*(\overline{\bbQ_l}|_{\mf_{a\boldsymbol{\nu},\Omega}})\cong {\pi_{\boldsymbol{\nu}}}_!(\overline{\bbQ_l}|_{\mf_{\boldsymbol{\nu},\Omega}})\cong L_{\boldsymbol{\nu}}$.

And by $\phi_0:\mathbf{a}^*L_{\boldsymbol{a\nu}}\cong L_{\boldsymbol{\nu}}$, the functor $\mathbf{a}^*$ could be restricted on $\mm_{a(\bfV)}$ to $\mm_{\bfV}$. 

The following definition is from \cite{lusztig2010introduction}.
\begin{definition}\label{Lusztig d}
    For an additive category $\mm_{\bfV}$, the Grothendieck group $\mk(\widetilde{\mm_{\bfV}})$ of $\widetilde{\mm_{\bfV}}$ is defined as follows:
    
    $\bullet$ For $(A,\phi)\cong (A',\phi')$, their image $[(A,\phi)]=[(A',\phi')]$ in the Grothendieck group.

    $\bullet$ $[(A,s\phi)]=s[(A,\phi)]$, $s\in \overline{\bbq_l}$.

    $\bullet$ $[(A,\phi)[n]]=v^n[(A,\phi)]$.

    $\bullet$ $[(A\oplus B,\phi\oplus \psi)]=[(A,\phi)]+[(B,\psi)]$.

    $\bullet$ If $(A,\phi)$ is a traceless element in $\widetilde{\cD}$, then $[(A,\phi)]=0$.
\end{definition}
By \cite[Proposition 12.5.2]{lusztig2010introduction}, for any fixed element $P \in \mP_{\bfV}$ under the action of $\mathbf{a}^*$, there exists a unique $\phi_P$ such that $(\mathbf{D}(P,\phi_P)) \cong (P,\phi_P)$ up to a sign $\pm1$.
Let $\mathcal{A} = \bbZ[v, v^{-1}]$. According to \cite[Proposition 12.6.3]{lusztig2010introduction}, the $\mathcal{A}$-module spanned by $[(P, \phi_P)]$, where $P \in \mP_{\bfV}$, is equal to the $\mathcal{A}$-module spanned by $[(L_{\boldsymbol{\nu}}, \phi_0)]$ for all $\boldsymbol{\nu}$ such that $a\boldsymbol{\nu} = \boldsymbol{\nu}$.

In \cite[12.6.4]{lusztig2010introduction}, Lusztig denotes the set ${[P, \phi_P] \mid P \in \mP_{\bfV}}$ (or ${\pm[P, \phi_P] \mid P \in \mP_{\bfV}}$ when $o$ is odd) by $\mathcal{B}_{\nu} = \mathcal{B}'_{\nu} \cup -\mathcal{B}'_{\nu}$, referring to it as the signed basis, where $\mathcal{B}'_{\nu}$ is a basis of $\mk(\widetilde{\mm_{\bfV}})$.

Furthermore, Lusztig also defines induction functors
 $$\Ind_{\nu',\nu''}^{\nu}:\cD_{\bfG_{\bfV'}}^b(\bfE_{\bfV'})\times\cD_{\bfG_{\bfV''}}^b(\bfE_{\bfV''})\rightarrow \cD_{\bfG_{\bfV}}^b(\bfE_{\bfV})$$ where the dimension vectors of the $I-$graded vector spaces $\bfV, \bfV', \bfV''$ are $\nu, \nu', \nu''$, respectively. Define $\bfE_{\bfV}'=\{(x,\bfW,\rho_1,\rho_2)|x\in \bfE_{\bfV},\bfW\cong \bfV''\text{as I-graded space}, \rho_1:\bfV/\bfW\cong \bfV', \rho_2: \bfW\cong \bfV''\}$, $\bfE_{\bfV}'=\{(x,\bfW)|x\in \bfE_{\bfV},\bfW\cong \bfV''\text{as I-graded space}\}$, with the maps defined by $p_1(x,\bfW,\rho_1,\rho_2)=(\rho_1(x|_{\bfV/\bfW})\rho_1^{-1},\rho_2(x|_{\bfW})\rho_2^{-1})$, $p_2(x,\bfW,\rho_1,\rho_2)=(x,\bfW)$ and $p_3(x,\bfW)=x$.
$$\begin{tikzcd}
\bfE_{\bfV'}\times \bfE_{\bfV''} & \bfE_{\bfV}' \arrow[l, "p_1"'] \arrow[r, "p_2"] & \bfE_{\bfV}'' \arrow[r, "p_3"] & \bfE_{\bfV}
\end{tikzcd}$$
We denote the dimension of fibers of the smooth morphisms $p_1$ and $p_2$ by $d_1$ and $d_2$, respectively. Define the induction functor as
\[
\Ind_{\nu',\nu''}^{\nu} := {p_3}_! {p_2}_b p_1^* [d_1 - d_2] \left( \frac{d_1 - d_2}{2} \right).
\]
Since $\mathbf{a}^* {p_3}_! \cong {p_3}_! \mathbf{a}^*$, there exists a natural transformation
\[
\mathbf{a}^* \Ind_{\nu',\nu''}^{\nu} \rightarrow \Ind_{\nu',\nu''}^{\nu} \mathbf{a}^*.
\]
Moreover, Lusztig proves in \cite[Lemma 12.3.2]{lusztig2010introduction} that
\[
\Ind_{\nu',\nu''}^{\nu} \left( (L_{\boldsymbol{\nu'}}, \phi_0) \boxtimes (L_{\boldsymbol{\nu''}}, \phi_0) \right) = (L_{\boldsymbol{\nu'\nu''}}, \phi_0).
\]
Hence, the induction functor $\Ind$ descends to a well-defined operation
\[
\Ind : \widetilde{\mm_{\bfV'}} \times \widetilde{\mm_{\bfV''}} \rightarrow \widetilde{\mm_{\bfV}}.
\]

This induction functor induces a multiplication $\cdot$ on the graded direct sum
\[
\tilde{\mk}=\bigoplus_{\bfV} \mk(\widetilde{\mm_{\bfV}}).
\]

In \cite[4.4]{lusztig1998canonical}, Lusztig provides the following formula. For 
\[
b = [(P_b, \phi_{P_b})] \in \mathcal{B}_{\nu}, \quad
b' = [(P_{b'}, \phi_{P_{b'}})] \in \mathcal{B}_{\nu'}, \quad
b'' = [(P_{b''}, \phi_{P_{b''}})] \in \mathcal{B}_{\nu''}
\]
we take $\alpha_{b}=\Hom\left( \phi_{P_b}, {^p}\mathbf{H}^j \left( \Ind_{\nu', \nu''}^{\nu}(\phi_{P_{b'}} \boxtimes \phi_{P_{b''}}) \right) \right)\in \End_{\bbC}\left(\Hom\left( P_b, {^p}\mathbf{H}^j \left( \Ind_{\nu', \nu''}^{\nu}(P_{b'} \boxtimes P_{b''}) \right) \right)\right)$. %and $\theta_{\lambda}=\Hom(\phi_{P_{b'}}\boxtimes \phi_{P_{b''}}, ^p\mathbf{H}^j(\mathbf{Res}_{\nu}^{\nu',\nu''}\phi_{P_b}))\in\End_{\bbC}(\Hom(P_{b'}\boxtimes P_{b''}, ^p\mathbf{H}^j(\mathbf{Res}_{\nu}^{\nu',\nu''}P_{b})))$.
We denote $m_{b',b''}^{b,j}=\operatorname{tr}\left( \alpha_b, \Hom\left( P_b, {^p}\mathbf{H}^j \left( \Ind_{\nu', \nu''}^{\nu}(P_{b'} \boxtimes P_{b''}) \right) \right) \right)$.
\begin{proposition}\cite{lusztig1998canonical}[6.3]\label{prop L}
    With the notations above and $\nu = \nu' + \nu''$, we have:
\begin{equation} \label{1}
b' \cdot b'' = \sum_{b \in \mathcal{B}_{\nu}',j\in \bbZ} m_{b',b''}^{b,j} v^{-j} b,
\end{equation}

%Similarly, \[\begin{aligned}&\mathbf{Res}_{\nu}^{\nu',\nu''}(b)\\&=\sum_{b',b''}tr(\theta_{\lambda},\Hom(P_{b'}\boxtimes P_{b''}, ^p\mathbf{H}^j(\mathbf{Res}_{\nu}^{\nu',\nu''}P_{b})))v^{-j}b'\otimes b'',\end{aligned}\]And we denote that $m_{b,j}^{b',b''}=tr(\theta_{\lambda},\Hom(\phi_{P_{b'}}\boxtimes \phi_{P_{b''}}, ^p\mathbf{H}^j(\mathbf{Res}_{\nu}^{\nu',\nu''}\phi_{P_{b}}))).$

\end{proposition}

Now for category $\mathcal{D}^b_{G_{\mathbf{V}}}(\mathbf{E}_{a(\mathbf{V}), \Omega})$ and the periodic functor
$$
\mathbf{a}^*: \mathcal{D}^b_{G_{\mathbf{V}}}(\mathbf{E}_{a(\mathbf{V}), \Omega}) \to \mathcal{D}^b_{G_{\mathbf{V}}}(\mathbf{E}_{\mathbf{V}, \Omega}).
$$

From now on, we will consider the mixed Weil complexes category $\mathcal{D}^b_{m,r, G_{\mathbf{V}}}(\mathbf{E}_{\mathbf{V}, \Omega})$ whose object has form like $(F,\phi),$ where $\phi: \mathbf{Fr^r}^*F\rightarrow F$ is its Weil structure. And if $a(\mathbf{V}) = \mathbf{V}$, $\mathbf{a}^*$ is a periodic functor on $\mathcal{D}^b_{m,r, G_{\mathbf{V}}}(\mathbf{E}_{\mathbf{V}, \Omega})$.
\begin{definition}\label{modified d}
The modified Grothendieck group $\mk_{\bfV}^r$ of $\widetilde{\mathcal{D}^b_{m,r,G_{\mathbf{V}}}(\mathbf{E}_{\mathbf{V}, \Omega})}$ when $\bfV=a(\bfV)$ is defined as follows:
\begin{itemize}
    \item[$\bullet$] For $(A, \varphi_A) \in \widetilde{\mathcal{D}^b_{m,r, G_{\mathbf{V}}}(\mathbf{E}_{\mathbf{V}, \Omega})}$, its image in the Grothendieck group of $\widetilde{\mathcal{D}^b_{m,r,G_{\mathbf{V}}}(\mathbf{E}_{\mathbf{V}, \Omega})}$ is denoted as $[(A, \varphi_A)]$.
    \item[$\bullet$] If $(A, \varphi_A) \xrightarrow{f} (B, \varphi_B) \xrightarrow{g} (C, \varphi_C) \xrightarrow{+1} $ satisfies that $A\xrightarrow{f} B\xrightarrow{g} C \xrightarrow{+1}$ is a triangle in $\mathcal{D}^b_{G_{\mathbf{V}}}(\mathbf{E}_{\mathbf{V}, \Omega})$ which is suitable with Weil structure, then $[(B, \varphi_B)] = [(A, \varphi_A)] + [(C, \varphi_C)]$ in the modified Grothendieck group.
    \item[$\bullet$] $[(A, s\varphi_A)] = s[(A, \varphi_A)]$, for $s \in \overline{\mathbb{Q}}_l$.
    \item[$\bullet$] If $(A, \varphi_A)$ is a traceless element in $\widetilde{\mathcal{D}^b_{m, r, G_{\mathbf{V}}}(\mathbf{E}_{\mathbf{V}, \Omega})}$, then $[(A, \varphi_A)] = 0$.
    \item[$\bullet$] $[(A\otimes \overline{\mathbb{Q}_l}(k),\varphi_A)]=q^{-k}[(A,\varphi_A)]$, $k\in \bbz.$
\end{itemize}
\end{definition}
The Grothendieck group from Definition \ref{modified d} and \ref{Lusztig d} would identify when $v=-\sqrt{q}^{-1}$.

We regard the map $a\circ Fr^r$ as a new Frobenius map on $\mathbf{E}_{\mathbf{V}, \Omega}$, when $a(\bfV)=\bfV$. Let $aFr^r:=a\circ Fr^r$. And for $(A,\phi_A)\in \widetilde{\mathcal{D}^b_{m, r, G_{\mathbf{V}}}(\mathbf{E}_{\mathbf{V}, \Omega})}$ where $A=(A,F)\in \mathcal{D}^b_{m,r, G_{\mathbf{V}}}(\mathbf{E}_{\mathbf{V}, \Omega})$, it gives $\mathbf{a}^*\mathbf{Fr^r}^*(A)\xrightarrow{\mathbf{Fr^r}^*(\phi_A)\circ F} A$ a Weil structure of Frobenius map $aFr^r$. It induces Weil structure of Frobenius functor $\mathbf{a}^*\circ \mathbf{Fr^r}^*$ on all elements in $\widetilde{\mathcal{D}^b_{m,r, G_{\mathbf{V}}}(\mathbf{E}_{\mathbf{V}, \Omega})}$. For $x=aFr^r(x)$, $\mathbf{H}^i_x(\mathbf{Fr^r}^*(\phi_A)\circ F)\in \End_{\overline{\bbq_l}}(\mathbf{H}^i_x(A))$

With the notation above, we denote the trace map of $((A,F),\phi_A)$ with Frobenius map $aFr^r$ as $\chi^a:\mathbf{E}_{\mathbf{V}, \Omega}^{aFr^r}\rightarrow \bbC$, where $\chi^a(x)=\sum_{i\in \bbZ}(-1)^itr(\mathbf{H}^i_x(A),\mathbf{H}^i_x(\mathbf{Fr^r}^*(\phi_A)\circ F)).$
\begin{proposition}\label{po1}
    The trace map $\chi^a$ of Weil structures of Frobenius functor $\mathbf{a}^*\circ \mathbf{Fr^r}^*$ could be induced from $\widetilde{\mathcal{D}^b_{m,r, G_{\mathbf{V}}}(\mathbf{E}_{\mathbf{V}, \Omega})}$ to its modified Grothendieck group.
\end{proposition}
\begin{proof}
    We only need to proof that this trace map satisfies the relations in Definition 1.6.

    For $(A, \varphi_A) \xrightarrow{f} (B, \varphi_B) \xrightarrow{g} (C, \varphi_C) \xrightarrow{+1} $ satisfies that $A\xrightarrow{f} B\xrightarrow{g} C \xrightarrow{+1}$ is a triangle in $\mathcal{D}^b_{G_{\mathbf{V}}}(\mathbf{E}_{a(\mathbf{V}), \Omega})$ which is suitable with Weil structure. Thus $\chi^a_{(B,\varphi_B)}=\chi^a_{(C,\varphi_C)}+\chi^a_{(A,\varphi_A)}$.

    For $(A,k\varphi_A)$, its trace map is $k$ times trace map of $(A,\varphi_A)$.

    And for traceless element $(A,\varphi_A)$, $H^i(\varphi_A)_x$ is a traceless matrix, thus is 0 under trace map.

    And Weil structure of $A\otimes \overline{\mathbb{Q}_l}(k)$ is $k$ times Weil structure of $A$, thus we have the proof.
\end{proof}
In the following sections, we will only consider the symmetrizable cases of finite types. All of these cases are quivers of finite type with admissible automorphisms, and we denote the index set of $G_{\mathbf{V}}$-orbits of $\bfE_{\bfV}$ as $\Lambda_{\bfV}$ with partical order that $\lambda\leqslant\lambda'\in \Lambda_{\bfV}$ when corresponding orbit $\mo_{\lambda}\subset \overline{\mo_{\lambda'}}$.

\section{The coefficients between PBW basis and canonical basis}

For the reason that all of the $G_{\mathbf{V}}$-orbits $\mathcal{O}$, and for $x \in \mathcal{O}$, the stabilizer of $x$ in $G_{\mathbf{V}}$ is connected, the only kinds of local systems on $\mathcal{O}$ have the form $\oplus \overline{\mathbb{Q}}_l|_{\mathcal{O}}$. The only simple $G_{\mathbf{V}}$-equivariant perverse sheaf on $\mathcal{O}$ is $\overline{\mathbb{Q}}_l|_{\mathcal{O}}[\dim \mathcal{O}]$. For $\mathbf{E}_{\mathbf{V}, \Omega}$, its natural Frobenius map $Fr^r$ induces the decomposition
$$
\mathbf{E}_{\mathbf{V}, \Omega} \cong \mathbf{E}_{\mathbf{V}, \Omega}^{Fr^r} \times_{\text{Spec}(\mathbb{F}_q)} \text{Spec}(\overline{\mathbb{F}_q}),
$$
where $\mathbf{E}_{\mathbf{V}, \Omega}^{Fr^r}$ consists of the fixed points of $Fr^r$, and we have the map

$$\mathbf{E}_{\mathbf{V}, \Omega}\xrightarrow{p_{\bfV}} \mathbf{E}_{\mathbf{V}, \Omega}^{Fr^r}$$ which is induced by $\text{Spec}(\overline{\mathbb{F}_q})\rightarrow \text{Spec}(\mathbb{F}_q)$. By \cite{Pramod-2021}[Lemma 5.3.8], let

$$
\text{egf}=p_{\bfV}^*: \mathcal{D}^b_{c,G_{\bfV}^{Fr^r}}(\mathbf{E}_{\mathbf{V}, \Omega}^{Fr^r}) \to \mathcal{D}^b_{m,r,G_{\bfV}}(\mathbf{E}_{\mathbf{V}, \Omega}).
$$

The intersection cohomology complex $(IC(\mathcal{O}_{\lambda}, \overline{\mathbb{Q}}_l)$ with the Weil structure $\xi_{\lambda}$ is given by 
$$
(IC(\mathcal{O}_{\lambda}, \overline{\mathbb{Q}}_l), \xi_{\lambda}) = \text{egf}(IC(\mathcal{O}_{\lambda}^{Fr^r}, \overline{\mathbb{Q}}_l))=(j_{\lambda})_{!*}(\text{egf}(\overline{\mathbb{Q}}_l|_{\mathcal{O}_{\lambda}})[\dim \mo_{\lambda}](\frac{\dim \mo_{\lambda}}{2})),
$$
Similarly, $(C_{\lambda}, \iota_{\lambda})$ for standard sheaf complex $C_{\lambda}$ with Weil structure $\iota_{\lambda}$, that is
$$
(C_{\lambda}, \iota_{\lambda}) = (j_{\lambda})_!(\text{egf}(\overline{\mathbb{Q}}_l|_{\mathcal{O}_{\lambda}}))[\dim \mathcal{O}_{\lambda}](\frac{\dim \mathcal{O}_{\lambda}}{2}),
$$
where $\lambda \in \Lambda$, and $j_{\lambda}: \mathcal{O}_{\lambda} \to \mathbf{E}_{\mathbf{V}, \Omega}$ is the embedding. Now, we choose $\gamma_{\lambda}[-\dim \mathcal{O}_{\lambda}]: \mathbf{a}^*\overline{\mathbb{Q}}_l|_{\mathcal{O}_{\lambda}} \to \overline{\mathbb{Q}}_l|_{\mathcal{O}_{\lambda}}$ such that the trace map of $\mathbf{Fr^r}^*\gamma_{\lambda}[-\dim \mathcal{O}_{\lambda}] \circ \xi_{\lambda}$ is characteristic map on $\mo_{\lambda}$ that is $$f(x)=\left\{
\begin{array}{l}
1, \quad x\in \mo_{\lambda} \\
0, \quad otherwise
\end{array}
\right.$$ Actually $\gamma_{\lambda}[-\dim \mathcal{O}_{\lambda}]$ is identity.

Furthermore, $\psi_{\lambda}:=(j_{\lambda})_![\dim \mathcal{O}_{\lambda}](\gamma_{\lambda}[-\dim \mathcal{O}_{\lambda}]): \mathbf{a}^*C_{\lambda} \to C_{\lambda}$. We denote that $\omega$ is the $n$-th root of 1.

For the reason that $IC(\mo_{\lambda},\overline{\bbq_l})$ is Lusztig simple perverse sheaf as in Definition 1.1, and by \cite{lusztig1998canonical}, there exists $\alpha_{\lambda}: \mathbf{a}^*IC(\mo_{\lambda},\overline{\bbq_l})\rightarrow IC(\mo_{\lambda},\overline{\bbq_l})$ satisfies the following properties:\\
$\bullet$ $\alpha_{\lambda}\circ \mathbf{a}^*\alpha_{\lambda}\circ ... \circ (\mathbf{a}^*)^{n-1}\alpha_{\lambda}=id$,\\
$\bullet$ $\mathbf{Fr^r}^*(\alpha_{\lambda})\circ \xi_{\lambda}=\mathbf{a}^*(\xi_{\lambda})\circ \alpha_{\lambda}$,\\
$\bullet$ $j_{\lambda}^*(\alpha_{\lambda})=id|_{\overline{\bbQ_l}[\dim\mo_{\lambda}]|_{\mo_{\lambda}}}$.

\begin{lemma}\label{lemma 1}
    The $\mathbb{Z}[\omega,\sqrt{q},\sqrt{q}^{-1}]-$group generated by $[IC(\mathcal{O}_{\lambda}, \overline{\mathbb{Q}}_l), \alpha_{\lambda}]$ in the modified Grothendieck group is free with basis $[IC(\mathcal{O}_{\lambda}, \overline{\mathbb{Q}}_l), \alpha_{\lambda}]$, and $\mathbb{Z}[\omega,\sqrt{q},\sqrt{q}^{-1}]-$group generated by $[C_{\lambda'}, \psi_{\lambda'}]$ in the modified Grothendieck group is free with basis $[C_{\lambda'}, \psi_{\lambda'}]$.
\end{lemma}
\begin{proof}
    By Proposition~\ref{po1} under trace maps, these elements form a set of basis.
\end{proof}
\begin{lemma}\label{lemma 2}
For simple perverse sheave $IC(\mo_{\lambda},\overline{\bbq_l})$, and $\bbF_q-$point $x\in \mo_{\lambda'}\subset \overline{\mo_{\lambda}}$, 

1, $\bfH^i_x(IC(\mo_{\lambda},\overline{\bbq_l}))=0$ when $i-\dim\mo_{\lambda}$ is odd

2, $\bfH^i_x(IC(\mo_{\lambda},\overline{\bbq_l}))=0$ when $i>-\dim\mo_{\lambda'}, \lambda'\not=\lambda$.

3, $\dim\bfH^{-\dim \mo_{\lambda}}_x(IC(\mo_{\lambda},\overline{\bbq_l}))=1$, when $x\in \mo_{\lambda}$.
\end{lemma}
\begin{proof}
For $i-\dim\mo_{\lambda}$ is odd, then by \cite[Corollary 10.7(a)]{lusztig1990canonical}, $\bfH^i_x(IC(\mo_{\lambda},\overline{\bbq_l}))=0$.

And by \cite[Lemma 3.3.11]{Pramod-2021} for $IC(\mo_{\lambda},\overline{\bbq_l})$ and $\mo_{\lambda'}\subset \overline{\mo_{\lambda}}-\mo_{\lambda}$, denoted $j_{\lambda'}:\mo_{\lambda'}\rightarrow \overline{\mo_{\lambda}}$, then $j_{\lambda'}^*IC(\mo_{\lambda},\overline{\bbq_l})\in \cD_c^{\leq-\dim\mo_{\lambda'}-1}(\mo_{\lambda'},\overline{\bbq_l}).$  And $j_{\lambda}^*IC(\mo_{\lambda},\overline{\bbq_l})\cong \overline{\bbq_l}|_{\mo_{\lambda}}[\dim\mo_{\lambda}]$.
\end{proof}
For $\alpha_{\lambda}$ could restrict on $\bfH^i_x(IC(\mo_{\lambda},\overline{\bbq_l}))$ when $a(x)=x, x\in \mo_{\lambda'}$, and in this case $\alpha^n=id$. We denoted the eigenspace decomposition over $\bfH^i_x(IC(\mo_{\lambda},\overline{\bbq_l}))$ of $\alpha_\lambda$ as $\bfH^i_x(IC(\mo_{\lambda},\overline{\bbq_l}))=\oplus_{s=1}^n \bfV_{-i-\dim\mo_{\lambda'},s}^{\lambda'},$ where eigenvalue of $\alpha_{\lambda}$ on $\bfV_{-i-\dim\mo_{\lambda'},s}^{\lambda'}$ is $\omega^s.$ For we will also consider the action of $\xi_{\lambda}$ acting on $\bfH^i_x(IC(\mo_{\lambda},\overline{\bbq_l}))$, we always choose $x\in \mo_{\lambda'}^{aFr^r}\cap \mo_{\lambda'}^{Fr^r}$. For quiver is finite type, thus $\mo_{\lambda'}^{Fr^r}\cap \mo_{\lambda'}^{aFr^r}\not=\emptyset.$

\begin{thm}\label{theorem}
For $\lambda,\lambda'\in \Lambda_{\bfV}$, there exists $p_{\lambda,\lambda'}(v)=\Sigma_{i\in \bbZ,s=1,\cdots ,n}v^{i}\omega^s \dim \bfV^{\lambda'}_{i,s}\in \bbZ[v,v^{-1},\omega]$. having the following property:

For any $r\in \bbZ_{>0}$ and $q=p^r$, in $\mk^r_{\bfV}$, we have
$$
[IC(\mathcal{O}_{\lambda}, \overline{\mathbb{Q}}_l), \alpha_{\lambda}] = \sum_{\lambda'} p_{\lambda, \lambda'}(-\sqrt{q}^{-1}) [C_{\lambda'}, \psi_{\lambda'}].
$$

\end{thm}

\begin{proof}

By \cite[Corollary 10.7]{lusztig1990canonical}, all of the eigenvalue of $\xi_{\lambda}$ on $\bfH^i_x(IC(\mo_{\lambda},\overline{\bbq_l}))$ where $x$ is $Fr^r$ fixed are $q^{\frac{i}{2}}$, and $\mo_{\lambda}$ is smooth. For $j_{\lambda}^*IC(\mo_{\lambda},\overline{\bbq_l})$ perverses the property above, it is pointwise pure. And for the reason that $\mo_{\lambda}$ is smooth, and is $G_{\bfV}-$orbit. Then the property that $j_{\lambda}^*IC(\mo_{\lambda},\overline{\bbq_l})$ is pointwise pure of weight 0 could induce that $j_{\lambda}^*IC(\mo_{\lambda},\overline{\bbq_l})$ is pure of weight 0.

 $(j_{\lambda'}^{Fr^r})^*IC(\mo_{\lambda}^{Fr^r},\overline{\bbq_l})$ is pure of weight zero, where $j_{\lambda'}^{Fr^r}:\mo_{\lambda'}^{Fr^r}\rightarrow \bfE_{\bfV,\Omega}^{Fr^r}$. Thus we also have that $j_{\lambda'}^*IC(\mo_{\lambda},\overline{\bbq_l})$ is semisimple perverse sheaves on $\mo_{\lambda'}$ with weight 0 by \cite{BBD}. Moreover eigenvalues of $\xi_{\lambda}$ on $\bfH^i_x(j_{\lambda'}^*IC(\mo_{\lambda},\overline{\bbq_l}))$ where $x$ is $Fr^r$ fixed point are all $q^{\frac{i}{2}}$. 

$$j_{\lambda'}^*IC(\mo_{\lambda},\overline{\bbq_l})\cong \oplus_r \overline{\bbq_{l}}|_{\mo_{\lambda'}}[\dim \mo_{\lambda'}+r](\frac{\dim \mo_{\lambda'}+r}{2})\boxtimes \bfV^{\lambda'}_r$$ where $\bfV^{\lambda'}_r$ is vector space with Weil complex on a point and $\xi_{\lambda}$ acts on $\bfV^{\lambda'}_r$ as an uniponent map.

Now we consider the triangle $j_!j^*\mf\rightarrow \mf \rightarrow i_!i^*\mf\xrightarrow{+1}$, where $\mf\in \mathcal{D}^{b}(\supp(\mf))$ and $\bfU$ is open in $\supp(\mf)$, $\bfV=\bfU^c$, $j:\bfU\rightarrow \supp(\mf)$. Weil structure of $\mf$ could be induced by $i^*$ where $i: \bfV\rightarrow \supp(\mf)$.

Now we choose that $\mf=IC(\mo_{\lambda},\overline{\bbq_l})$, $\bfU=\mo_{\lambda}$, it gives a triangle that $C_{\lambda}\rightarrow IC(\mo_{\lambda},\overline{\bbq_l})\rightarrow i_!i^*IC(\mo_{\lambda},\overline{\bbq_l})\xrightarrow{+1}$ which is suitable with the Weil structure.

And then we choose $\mf=i^*IC(\mo_{\lambda},\overline{\bbq_l})$, and open orbit $\mo_{\lambda'}$ in $\supp(\mf)$ as new $\bfU$. Then we would have that $(j_{\lambda'})_!(j_{\lambda'})^*i^*IC(\mo_{\lambda},\overline{\bbq_l})\rightarrow i^*IC(\mo_{\lambda},\overline{\bbq_l})\rightarrow (i_\lambda')_!(i_{\lambda'})^*i^*IC(\mo_{\lambda},\overline{\bbq_l})\xrightarrow{+1}$, where we denote $\mo_{\lambda'}\rightarrow \supp(i^*IC(\mo_{\lambda},\overline{\bbq_l}))$ also $j_{\lambda'}$ and $i_{\lambda'}:\bfV\rightarrow \supp(i^*IC(\mo_{\lambda},\overline{\bbq_l}))$

$(j_{\lambda'})^*i^*IC(\mo_{\lambda},\overline{\bbq_l})=(j_{\lambda'})^*IC(\mo_{\lambda},\overline{\bbq_l})\cong \oplus_r \overline{\bbq_{l}}|_{\mo_{\lambda'}}[\dim \mo_{\lambda'}+r](\frac{\dim \mo_{\lambda'}+r}{2})\boxtimes \bfV^{\lambda'}_r$

Then by induction that we could choose $\mf=(i_{\lambda'})^*i^*IC(\mo_{\lambda},\overline{\bbq_l})$ and so on, there we would have that in the Grothendieck group of $\mathcal{D}^{b}_{m,r,G_{\bfV}}(\bfE_{\bfV,\Omega})$  
$$[IC(\mo_{\lambda},\overline{\bbq_l})]=\oplus_{\lambda',r} \dim \bfV^{\lambda'}_r[C_{\lambda'}[r](\frac{r}{2})]$$ and all $\mo_{\lambda'}\not\subset \overline{\mo_{\lambda}}$ $\bfV^{\lambda'}_r=0$, $\bfV^{\lambda}_r=0$ for all $r\not=0$, $\dim \bfV_0^{\lambda}=1$.

For the reason that $j_!j^*\mf\rightarrow \mf \rightarrow i_!i^*\mf\xrightarrow{+1}$ communites with $\alpha_{\lambda}$, thus the morphism $\alpha_{\lambda}$ could be induced on $j_{\lambda'}^*IC(\mo_{\lambda},\overline{\bbq_l})$ noted as $\phi_{\lambda'}$ and preserve the two properties above, and if $a(\mo_{\lambda'})\not=\mo_{\lambda'}$, it is easy to show that this kind of elements gives traceless part.

And then $\phi_{\lambda'}=(\gamma_{\lambda'}[r],\tilde{\phi})$ on $\overline{\bbq_{l}}|_{\mo_{\lambda'}}[\dim \mo_{\lambda'}+r](\frac{\dim \mo_{\lambda'}+r}{2})\boxtimes \bfV^{\lambda'}_r$, where $\tilde{\phi}$ is morphism on $\bfV^{\lambda'}_{r}$ and $\tilde{\phi}^n=id$, and for the reason that $\tilde{\phi}$ is communitative with the morphism $\xi_{\lambda}$ on $\bfV^{\lambda'}_{r}$ thus we could consider the eigenspace $\bfV^{\lambda'}_{r,s}$ with eigenvalue $\omega^s$ of the map $\tilde{\phi}$, it provides a direct decomposition $\oplus_{s=0}^{n-1}\bfV^{\lambda'}_{r,s}$ of $\bfV^{\lambda'}_{r}$ where $\xi_{\lambda}$ is stable on $\bfV^{\lambda'}_{r,s}$. Thus $[(\overline{\bbq_{l}}|_{\mo_{\lambda'}}[\dim \mo_{\lambda'}+r](\frac{\dim \mo_{\lambda'}+r}{2})\boxtimes \bfV^{\lambda'}_r,\phi_{\lambda'})]=\Sigma_{s=0}^{n-1}\omega^s \dim \bfV^{\lambda'}_{r,s} [\overline{\bbq_{l}}|_{\mo_{\lambda'}}[\dim \mo_{\lambda'}+r](\frac{\dim \mo_{\lambda'}+r}{2}),\gamma_{\lambda'}[r]]$.

Then $[IC(\mo_{\lambda},\overline{\bbq_l}),\alpha_{\lambda}]=\Sigma_{\lambda',r}\Sigma_{s=0}^{n-1}\omega^s \dim \bfV^{\lambda'}_{r,s} [C_{\lambda'}[r](\frac{r}{2}),\psi_{\lambda'}[r]]$ and is suitable with trace map of Frobenius functor  $(a\circ \mathbf{Fr^r})^*$, which means that for $x\in \mo_{\lambda'}\subset \overline{\mo_{\lambda}}$ and fixed under $Fr^r$ and $aFr^r$, $$\Sigma_{i\in \bbZ}(-1)^i tr(\mathbf{Fr^r}^*(\alpha_{\lambda})\circ \xi_{\lambda}, \bfH_x^i(IC(\mo_{\lambda},\overline{\bbq_l})))=\sum_{r\in \bbZ}\sum_{s=0}^{n-1}\omega^s \dim \bfV^{\lambda'}_{r,s} (-\sqrt{q})^{(-r-\dim \mo_{\lambda'})}$$
And for $\mo_{\lambda'}\not\subset \overline{\mo_{\lambda}}$, $\bfV_{r,s}^{\lambda'}=0$ and $\dim \bfV^{\lambda}_{0,0}=1$. $\bfV^{\lambda}_{s,r}=0$ if $s$ or $r$ is not zero. 
\end{proof}
\begin{corollary}
    For $\lambda', \lambda\in \Lambda_{\bfV}$, we have the following properties for $p_{\lambda,\lambda'}(v)$,

    1, $p_{\lambda, \lambda'}(v) = 0$ if $\lambda'\nleqslant \lambda$.

    2, $p_{\lambda, \lambda'}(v)\in v\bbZ[\omega,v]$ if $\lambda'\not=\lambda$, and $p_{\lambda, \lambda} = 1.$

    3, $p_{\lambda, \lambda'}(v) \in v^{\dim\mo_{\lambda'}-\dim\mo_{\lambda}}\mathbb{Z}[\omega,v^2].$

    Thus the $\mathbb{Z}[\omega,\sqrt{q},\sqrt{q}^{-1}]$ free module generated by $[IC(\mathcal{O}_{\lambda}, \overline{\mathbb{Q}}_l), \alpha_{\lambda}]$ is the same as $\mathbb{Z}[\omega,\sqrt{q},\sqrt{q}^{-1}]$ free module generated by $[C_{\lambda'}, \psi_{\lambda'}]$ in $\mk_{\bfV}^r$.
\end{corollary}
\begin{proof}
    By Lemma~\ref{lemma 2}, we have that $\bfH^i_x(IC(\mo_{\lambda},\overline{\bbq_l}))=0$ if $\lambda'\nleqslant \lambda$, $\bfH^i_x(IC(\mo_{\lambda},\overline{\bbq_l}))=0 $ when $i>-\dim\mo_{\lambda'}, \lambda'\not=\lambda$, thus $\bfV_{-i-\dim \mo_{\lambda'},s}=0$ when $i>-\dim\mo_{\lambda'}$, thus $\bfV_{r,s}=0$ when $r\leq 0$. And $\bfH^i_x(IC(\mo_{\lambda},\overline{\bbq_l}))=0$ if $i+\dim\mo_{\lambda}$ is odd, then $\bfV_{r,s}=0$ if $r+\dim \mo_{\lambda'}-\dim \mo_{\lambda}$ is odd. So we have the proof.
\end{proof}

We let matrix $\mathbf{P}_{\bfV}=(p_{\lambda,\lambda'}(v))_{\lambda,\lambda'\in \Lambda_{\bfV}}$. From last corollary and theorem \ref{theorem}, $\mathbf{P}_{\bfV}$ is an upper triangular matrix with all diagonal elements $1$ and $p_{\lambda,\lambda'}(v)\in \bbZ[v,v^{-1},\omega]$. Thus we could calculate that its inverse matrix $\mathbf{Q}_{\bfV}=(q_{\lambda,\lambda'}(v))_{\lambda,\lambda'}$ is also an upper triangular matrix with all diagonal elements $1$ and $q_{\lambda,\lambda'}(v)\in \bbZ[v,v^{-1},\omega]$. Specificly, in $\mk^r_{\bfV}$ for $\lambda\in \Lambda_{\bfV}$, $[C_{\lambda},\psi_{\lambda}]=\sum_{\lambda'\in \Lambda_{\bfV}}q_{\lambda,\lambda'}(-\sqrt{q}^{-1})[IC(\mo_{\lambda'},\overline{\bbq_l}),\alpha_{\lambda'}]$.

By Theorem~\ref{theorem}, the modified Grothendieck group is linearly isomorphic to the function space corresponding to its image under the trace map.

\begin{definition}
    For finite dimensional hereditary algebra $A$ over field $\bbF_q$, and $A-$module $M,N,L$, denote $g_{MN}^{L}(q)$ the number of submodules $L'\subset L$ satisfying $L/L'\cong M, L'\cong N$.
\end{definition}
\begin{definition}
$aFr^r$-stable representation category of $Q$ is denoted as $\operatorname{rep}^{aFr^r}(Q)$. More precisely, its objects are triples like $(V, x, F_V)$, where $(V, x) \in \rep_k(Q)$ and $aFr^r_V: V \rightarrow V$ is Frobenius map on $\overline{\bbF_q}$-space $V$ 
 satisfying:  

{\rm{(a)}} For any $i \in I$, $aFr^r_V(V_i) = V_{a(i)}$;

{\rm{(b)}} For any $h \in H$, $aFr^r_V x_h = x_{a(h)} aFr^r_V: V_{s(h)} \rightarrow V_{t(a(h))}$.  

Morphism $(V, x, aFr^r_V) \rightarrow (V', x', aFr^r_{V'})$ is morphism $(f_i)_{i \in I}: (V, x) \rightarrow (V', x')$ in $\rep_k(Q)$, satisfying for any $i \in I$, $f_{a(i)} aFr^r_V = aFr^r_{V'} f_i$.
\end{definition}

For the pair $(C_{\lambda}, \psi_{\lambda})$, its trace map $\chi^a_{(C_{\lambda}, \psi_{\lambda})}$ of Frobenuis functor $\mathbf{a}^*\circ \mathbf{Fr^r}^*$ defines a character function up to $-\sqrt{q}^{-1}$ times on the $\bfG_{\bfV}^{aFr^r}$-orbit of an $aFr^r$-stable representation in $\bfE_{\bfV}$.

We denote by $A$ the $\bbF_q$-algebra corresponding to the Cartan matrix $C$, with its module category written as $\operatorname{mod}(A)$. The category of $aFr^r$-stable representations of the quiver $Q$ is denoted by $\operatorname{rep}^{aFr^r}(Q)$.

\begin{thm}\cite[Theorem 3.2, Section 9]{deng2006frobenius}\label{22}
    There is equivalence between categories,
    $$\operatorname{rep}^{aFr^r}(Q)\cong \operatorname{mod}(A)$$.
\end{thm}
Thus, we can always index the character function of a $\bfG_{\bfV}^{aFr^r}$-orbit of an $aFr^r$-stable representation by its corresponding module in $\operatorname{mod}(A)$.

By \cite{Fang2025}[Lemma 4.4],\cite{deng2006frobenius}[ Theorem 6.5], we know that for any finite dimensional hereditary algebra of finite type over field $\bbF_q$, there exists quiver with automorphism $Q=(I,H,s,t,a)$ of finite type, such that there is a bijection of the set of isomorphism class of $\mod A$ and the set of $\mo^{aFr^r}$ where $\mo$ is $\bfG_{\bfV}-$orbit in $\bfE_{\bfV}$ for some $\bfV$. Thus if we denote $\mo_{\lambda}^{aFr^r},\mo_{\lambda'}^{aFr^r},\mo_{\lambda''}^{aFr^r}$ as the corresponding elements of $A-$modules $L,M,N$, we could also denote $g_{MN}^{L}$ as $g_{\lambda', \lambda''}^{\lambda}$.

Let $X^{aFr^r}$ and $G^{aFr^r}$ be the fixed point sets of $X$ and $G$ under their Frobenius maps ${aFr^r}$ respectively, we denote by $\tilde{\cH}_{G^{aFr^r}}(X^{aFr^r})$ the $\bbC$-vector space of $G^{aFr^r}$-invariant functions on $X^{aFr^r}$. For any $G$-equivariant morphism $f:X\rightarrow Y$ which respects the $\bbF_q$-structures, its restriction to $X^{aFr^r}$ is still denoted by $f:X^{aFr^r}\rightarrow Y^{aFr^r}$, and there are $\bbC$-linear maps
\begin{align*}
f^*:\tilde{\cH}_{G^{aFr^r}}(Y^{aFr^r})&\rightarrow \tilde{\cH}_{G^{aFr^r}}(X^{aFr^r}) &f_!: \tilde{\cH}_{G^{aFr^r}}(X^{aFr^r})&\rightarrow \tilde{\cH}_{G^{aFr^r}}(Y^{aFr^r})\\
\varphi&\mapsto (x\mapsto f(\varphi(x))), & \psi&\mapsto(y\mapsto \sum_{x\in f^{-1}(y)}\psi(x)).
\end{align*}

\begin{proposition}\cite{Pramod-2021}[Theorem 5.3.13]
    We have the following equations, $\chi^a(f^*F)=f^*\chi^a(F)$, $\chi^a(f_!G)=f_!(\chi^a(G))$, for $F\in \widetilde{\cD_{m,r,G}^b(Y)},G\in \widetilde{\cD_{m,r,G}^b(X)}$.
\end{proposition}
Then by calculation, $$\chi^a(\Ind_{\nu',\nu''}^{\nu}((C_{\lambda'},\psi_{\lambda'})\boxtimes (C_{\lambda''},\psi_{\lambda''})))=(-\sqrt{q})^{\sum_{i\in I}\nu_i'\nu_i''+\sum_{h\in H}\nu_{s(h)}'\nu_{t(h)}''-\dim\mo_{\lambda'}-\dim\mo_{\lambda''}}\sum_{\lambda\in \Lambda_{\bfV}}g_{\lambda',\lambda''}^{\lambda}(q)1_{\mo^{aFr^r}_{\lambda}}.$$

It is obvious that the rightside is equal to$$(-\sqrt{q})^{\sum_{i\in I}\nu_i'\nu_i''+\sum_{h\in H}\nu_{s(h)}'\nu_{t(h)}''-\dim\mo_{\lambda'}-\dim\mo_{\lambda''}}\sum_{\lambda\in \Lambda_{\bfV}}g_{\lambda',\lambda''}^{\lambda}(q)(-\sqrt{q})^{\dim\mo_{\lambda}}\chi^a(C_{\lambda},\psi_{\lambda}).$$

From now on, we denote $[(IC(\mo_{\lambda},\overline{\bbq_l}),\alpha_{\lambda})]$ and the one with Weil structure $\xi_{\lambda}$ both $b_{\lambda}$, for we could distinguish them with the Grothendieck group they belong to. By the Proposition\ref{prop L}, in $\tilde{\mk}$, we have $$b_{\lambda}\cdot b_{\lambda'}=\sum_{\lambda}h_{\lambda',\lambda''}^{\lambda}(v)b_{\lambda},$$ where $h_{\lambda',\lambda''}^{\lambda}(v)\in \bbZ[v,v^{-1},\omega]$. And for when we choose $v=-\sqrt{q}^{-1}$, and Weil structure $\xi_{\lambda}$ on $IC(\mo_{\lambda},\overline{\bbq_l})$, By \cite{lusztig1998canonical}[Theorem 6.2] the equation above remains true in modified grothendieck group $\mk^r=\oplus_{\bfV}\mk^r_{\bfV}$ when the multiplication on $\mk^r$ is induced by the induction functor $\Ind$. By acting trace map $\chi^a$, we have that $\chi^a(\Ind_{\nu',\nu''}^{\nu}b_{\lambda'}\boxtimes b_{\lambda''})=\sum_{\lambda\in \Lambda_{\bfV}}h_{\lambda',\lambda''}^{\lambda}(-\sqrt{q}^{-1})\chi^a(b_{\lambda})$

\begin{corollary}\label{2}
   There exists $f_{\lambda',\lambda''}^{\lambda}(v)\in \bbZ[v^{-2}]$, such that for any $q=p^r$, $f_{\lambda',\lambda''}^{\lambda}(-\sqrt{q}^{-1})=g_{\lambda',\lambda''}^{\lambda}(q)$.
\end{corollary}
\begin{proof}
   For we already have that $$\begin{aligned}
       \chi^a(\Ind_{\nu',\nu''}^{\nu}((C_{\lambda'},\psi_{\lambda'})\boxtimes (C_{\lambda''},\psi_{\lambda''})))=&(-\sqrt{q})^{\sum_{i\in I}\nu_i'\nu_i''+\sum_{h\in H}\nu_{s(h)}'\nu_{t(h)}''-\dim\mo_{\lambda'}-\dim\mo_{\lambda''}}\\&\sum_{\lambda\in \Lambda_{\bfV}}g_{\lambda',\lambda''}^{\lambda}(q)(-\sqrt{q})^{\dim\mo_{\lambda}}\chi^a(C_{\lambda},\psi_{\lambda}).\end{aligned}$$
    And 
    $[C_{\lambda},\psi_{\lambda}]=\sum_{\lambda'\in \Lambda_{\bfV}}q_{\lambda,\lambda'}(-\sqrt{q}^{-1})b_{\lambda'}.$
    
    Thus
    $$\begin{aligned}   
    &\chi^a(\Ind_{\nu',\nu''}^{\nu}((C_{\lambda'},\psi_{\lambda'})\boxtimes (C_{\lambda''},\psi_{\lambda''})))=\sum_{\lambda_1\in \Lambda_{\bfV'},\lambda_2\in \Lambda_{\bfV''}}q_{\lambda',\lambda_1}(-\sqrt{q}^{-1})q_{\lambda'',\lambda_2}(-\sqrt{q}^{-1})\chi^a(\Ind_{\nu',\nu''}^{\nu}b_{\lambda_1}\boxtimes b_{\lambda_2})\\=&\sum_{\lambda_1\in \Lambda_{\bfV'},\lambda_2\in \Lambda_{\bfV''},\lambda_3\in \Lambda_{\bfV}}q_{\lambda',\lambda_1}(-\sqrt{q}^{-1})q_{\lambda'',\lambda_2}(-\sqrt{q}^{-1})h_{\lambda_1,\lambda_2}^{\lambda_3}(-\sqrt{q}^{-1})\chi^a(b_{\lambda_3})\\=&\sum_{\lambda_1\in \Lambda_{\bfV'},\lambda_2\in \Lambda_{\bfV''},\lambda_3,\lambda\in \Lambda_{\bfV}, }q_{\lambda',\lambda_1}(-\sqrt{q}^{-1})q_{\lambda'',\lambda_2}(-\sqrt{q}^{-1})h_{\lambda_1,\lambda_2}^{\lambda_3}(-\sqrt{q}^{-1})p_{\lambda_3,\lambda}(-\sqrt{q}^{-1})\chi^a((C_{\lambda},\psi_{\lambda})).&\end{aligned}$$

   Thus 
   $$\begin{aligned}      
  &(-\sqrt{q})^{\sum_{i\in I}\nu_i'\nu_i''+\sum_{h\in H}\nu_{s(h)}'\nu_{t(h)}''-\dim\mo_{\lambda'}-\dim\mo_{\lambda''}}g_{\lambda',\lambda''}^{\lambda}(q)(-\sqrt{q})^{\dim\mo_{\lambda}}\\=&\sum_{\lambda_1\in \Lambda_{\bfV'},\lambda_2\in \Lambda_{\bfV''},\lambda_3,\in \Lambda_{\bfV}, }q_{\lambda',\lambda_1}(-\sqrt{q}^{-1})q_{\lambda'',\lambda_2}(-\sqrt{q}^{-1})h_{\lambda_1,\lambda_2}^{\lambda_3}(-\sqrt{q}^{-1})p_{\lambda_3,\lambda}(-\sqrt{q}^{-1}) \end{aligned}$$
  By the properties of $q_{\lambda',\lambda_1}(v)$, $q_{\lambda'',\lambda_2}(v)$, $h_{\lambda_1,\lambda_2}^{\lambda_3}(v)$, and $p_{\lambda_3,\lambda}(v)$, as well as $g_{\lambda',\lambda''}^{\lambda}(q)\in \bbZ$, we have the proof.
\end{proof}
\section{Structure coefficients for bar involution on PBW basis}
For the PBW basis, we consider its bar involution, which corresponds to the Verdier duality on the modified Grothendieck group via trace map. This part of work is a slight generalization of Caldero and Reineke's work\cite{2004The}. 

By Theorem \ref{22}, we can always index the character function of a $\bfG_{\bfV}^{aFr^r}$-orbit of an $aFr^r$-stable representation by its corresponding module in $\operatorname{mod}(A)$. For example, we denote the character function associated with the $\bfG_{\bfV}^{aFr^r}$-orbit of an $aFr^r$-stable representation by $1_{\mo^{aFr^r}_M}$, where $M$ is the corresponding module. Here, $\mo^{aFr^r}_M$ denotes the $\bfG_{\bfV}^{aFr^r}$-orbit of a point $x$ such that $(V_M, x, aFr^r)$ corresponds to $M$, and $\nu_M$ is the dimension vector of $V_M$.

Assume $A$ is of finite type. Let $\{I_t\}_{1 \leq t \leq m}$ be the complete set of indecomposable $A$-modules, ordered so that
\[
\Hom_A(I_s, I_t) = 0 \quad \text{for } s > t.
\]

For $M, N \in \operatorname{mod}(A)$, suppose $N = \bigoplus_{t=1}^m N_t$ where each $N_t = I_t^{\oplus r_t}$. Let $F_{N_1, \dots, N_m}^M(q)$ denote the number of filtrations
\[
0 = M_m \subset M_{m-1} \subset \cdots \subset M_1 \subset M_0 = M
\]
satisfying $M_i / M_{i+1} \cong N_{i+1}$ for each $i$. By Corollary~\ref{2}, the function $F_{N_1, \dots, N_m}^M(q)$ is a polynominal. For any $M \in \operatorname{mod}(A)$, let $a_M := |\Aut_A(M)|$.

Now consider the bar involution of $1_{\mo^{aFr^r}_M}$, denoted by $\overline{1_{\mo^{aFr^r}_M}}$. According to \cite[Proposition 3.3]{2004The}, we have the following expression:
\begin{equation} \label{gs}
\overline{1_{\mo^{aFr^r}_M}} = \sum_{N \in \operatorname{mod}(A)} v^{\dim \mo^{aFr^r}_M - \dim \mo^{aFr^r}_N} \, \overline{F_{N_1, \dots, N_m}^M(q) \cdot \frac{\prod_{i=1}^m a_{N_i}}{a_M}} \cdot 1_{\mo^{aFr^r}_N}.
\end{equation}

We define the following varieties:
\begin{itemize}
    \item $\bfE_{\bfV,m}''$ is the variety
    \[
    \left\{ (x, W_1, \dots, W_{m-1}) \,\middle|\, 
    \begin{array}{l}
    x(W_i) \subset W_i, \; W_i \subset W_{i-1}, \\
    \dim W_{i-1} / W_i = \nu_{N_i}, \; x \in \bfE_{\bfV}
    \end{array}
    \right\};
    \]
    
    \item $\bfE_{\bfV,m}'$ is the variety
    \[
    \left\{ (x, W_1, \dots, W_{m-1}, \rho_1, \dots, \rho_m) \,\middle|\,
    \begin{array}{l}
    x(W_i) \subset W_i, \; W_i \subset W_{i-1}, \\
    \dim W_{i-1}/W_i = \nu_{N_i}, \; W_0 = \bfV, \; x \in \bfE_{\bfV}, \\
    \rho_i : W_{i-1}/W_i \xrightarrow{\sim} V_{N_i}
    \end{array}
    \right\}.
    \]
\end{itemize}

Define the maps:
\[
p_1^m(x, W_1, \dots, W_{m-1}, \rho_1, \dots, \rho_m) = \left( \rho_i(x|_{W_{i-1}/W_i}) \rho_i^{-1} \right)_{i=1}^m,
\]
\[
p_2^m(x, W_1, \dots, W_{m-1}, \rho_1, \dots, \rho_m) = (x, W_1, \dots, W_{m-1}),
\]
\[
p_3^m(x, W_1, \dots, W_{m-1}) = x.
\]
See the following graph,
$$\begin{tikzcd}
\bfE_{V_{N_1}}^{a{Fr^r}}\times \cdots \bfE_{V_{N_m}}^{a{Fr}^r} & {{\bfE_{\bfV,m}^{a{Fr}^r}}^{\prime}} \arrow[l, "p_1^m"'] \arrow[r, "p_2^m"] & {{\bfE_{\bfV,m}^{a{Fr}^r}}^{\prime\prime}} \arrow[r, "p_3^m"] & {\bfE_{\bfV,m}^{a{Fr}^r}}
\end{tikzcd}$$

Fix a filtration
\[
0 = \bfW_m \subset \cdots \subset \bfW_1 \subset \bfW_0 = \bfV,
\]
such that $\bfW_i$ is stable under  $a,Fr^r$ for each $i$ and $\bfW_{i-1}/\bfW_i \cong V_{N_i}$ with $a,Fr^r$ induced. We denote by $\bfP_{\nu} \subset \bfG_{\bfV}$ the subgroup preserving this filtration, and by $\bfU_{\nu} \subset \bfP_{\nu}$ its unipotent radical.

It is easy to see that, if we define
\[
F := \left\{ x \in \bfE_{\bfV}^{aFr^r} \,\middle|\, x(\bfW_i) \subset \bfW_i,\; \text{for } i = 1, \dots, m \right\},
\]
then we have the identifications:
\[
{\bfE_{\bfV,m}^{aFr^r}}^{\prime\prime} = F \times_{\bfP_{\nu}^{aFr^r}} \bfG_{\bfV}^{aFr^r}, \quad
{\bfE_{\bfV,m}^{aFr^r}}^{\prime} = F \times_{\bfU_{\nu}^{aFr^r}} \bfG_{\bfV}^{aFr^r}.
\]

Now, for a sequence of $aFr^r$-stable representations $(V_{N_i}, x_i, aFr^r)$, we define the following subsets of $F$:
\begin{itemize}
    \item $Y_N := \left\{ x \in F \,\middle|\, x|_{\bfW_{i-1}/\bfW_i} = x_i,\; \text{for } i = 1, \dots, m \right\}$,
    \item $Y_{\bar{N}} := \left\{ x \in F \,\middle|\, x|_{\bfW_{i-1}/\bfW_i} \in \bfG_{V_{N_i}}^{aFr^r}(x_i),\; \text{for } i = 1, \dots, m \right\}$.
\end{itemize}
\begin{proposition}
    There is an equation,$$F_{N_1,\cdots,N_m}^M(q)=\frac{a_M|Y_N\cap\mo_{M}^{aFr^r}|}{\Pi_{i=1}^ma_{N_i}|\bfU_{\nu}^{aFr^r}|}.$$
\end{proposition}
\begin{proof}
    For \begin{align*}
        F_{N_1,\cdots,N_m}^M(q)&=\frac{|(p_1^m)^{-1}(\mo_{M}^{aFr^r})\cap \bfG_{\bfV}^{aFr^r}\times_{\mathbf{P}_{\nu}^{aFr^r}}Y_{\bar{N}}|}{|\mo_{M}^{aFr^r}|}\\&=\frac{|\bfG_{\bfV}^{aFr^r}\times_{\mathbf{P}_{\nu}^{aFr^r}}(\mo_{M}^{aFr^r}\cap Y_{\bar{N}})|}{|\mo_{M}^{aFr^r}|}\\&=\frac{|\bfG_{\bfV}^{aFr^r}|\cdot|\mo_{M}^{aFr^r}\cap Y_{\bar{N}}|}{|\mathbf{P}_{\nu}|\cdot |\mo_{M}^{aFr^r}|}\\&=\frac{|\bfG_{\bfV}^{aFr^r}|\cdot|\mo_{M}^{aFr^r}\cap Y_{N}|\Pi_{i=1}^m|\mo_{N_i}^{aFr^r}|}{|\mathbf{P}_{\nu}^{aFr^r}|\cdot |\mo_{M}^{aFr^r}|}\\&=\frac{|\bfG_{\bfV}^{aFr^r}|\cdot|\mo_{M}^{aFr^r}\cap Y_{N}|\Pi_{i=1}^m|\mo_{N_i}^{aFr^r}|}{|\mathbf{U}_{\nu}^{aFr^r}|\cdot |\mo_{M}^{aFr^r}|\Pi_{i=1}^m|\bfG_{V_{Ni}}^{aFr^r}|}\\&=\frac{a_M|Y_N\cap\mo_{M}^{aFr^r}|}{\Pi_{i=1}^ma_{N_i}|\bfU_{\nu}^{aFr^r}|}.
    \end{align*}
    We have the proof.
\end{proof}
When we fix $N$ with an $aFr^r-$stable representation $(V_N,x,aFr^r)$, denote Lie algebra of $\bfG_{\bfV}$ as $\mathfrak{g}_{\nu}$ and consider the map $\phi$ as in \cite[2.5]{2004The},  $$\begin{tikzcd}
{\phi:\mathfrak{g}_{\nu}} \arrow[r]  & \bfE_{\bfV}                        \\
(g_i)_{i\in I} \arrow[r, maps to] & (g_jx_h-x_hg_i|h:i\rightarrow j)_{h\in \Omega}.
\end{tikzcd}$$
We denote Lie algebra of $\bfU_{\nu}$ as $\mathfrak{u}_{\nu}$, and $\phi(\mathfrak{u}_{\nu})$ as $E_{N}$.

By \cite[Proposition 2.5]{2004The}, and find the $aFr^r-$fixed space, $E_{N}^{aFr^r}$ is a tranversal slice of $\bfU_{\nu}^{aFr^r}$ on $Y_N$.

\begin{corollary}
There is an equation,
$$F_{N_1,\cdots,N_m}^M(q)=\frac{a_M|E_N^{aFr^r}\cap \mo^{aFr^r}_{M}|}{\Pi_{i=1}^ma_{N_i}}.$$
\end{corollary}
Apply this corollary to equation \ref{gs}, we will have the following theorem.
\begin{thm}
There is an equation,
$$\overline{1_{\mo^{aFr^r}_M}}=\sum_{N\in \mod{A}}v^{\dim \mo^{aFr^r}_{M}-\dim \mo^{aFr^r}_{N}}\overline{|E_N^{aFr^r}\cap \mo^{aFr^r}_{M}|}1_{\mo^{aFr^r}_N}.$$
\end{thm}
\begin{corollary}
By Theorem \ref{theorem}, $$\overline{1_{\mo^{aFr^r}_{M_{\lambda}}}}=(-\sqrt{q})^{-\dim\mo_{\lambda}}\sum_{\lambda',\lambda''}q_{\lambda,\lambda'}(-\sqrt{q})p_{\lambda',\lambda''}(-\sqrt{q}^{-1}))(-\sqrt{q})^{\dim\mo_{\lambda''}}1_{\mo^{aFr^r}_{M_{\lambda''}}}$$
\end{corollary}
\begin{proof}
 $$\begin{aligned}\overline{(-\sqrt{q})^{\dim\mo_{\lambda}}[C_{\lambda},\psi_{\lambda}]}=&(-\sqrt{q})^{-\dim\mo_{\lambda}}\sum_{\lambda'}\overline{q_{\lambda,\lambda'}(-\sqrt{q}^{-1}))}b_{\lambda'}\\=&(-\sqrt{q})^{-\dim\mo_{\lambda}}\sum_{\lambda',\lambda''}q_{\lambda,\lambda'}(-\sqrt{q}))p_{\lambda',\lambda''}(-\sqrt{q}^{-1}))[C_{\lambda''},\psi_{\lambda''}]\end{aligned}$$
\end{proof}
\subsection*{Acknowledgement} 
Y. Lan, Y. Wu and J. Xiao are supported by National Natural Science Foundation of China [Grant No. 12031007]
and National Natural Science Foundation of China [Grant No. 12471030].
\end{spacing}

\bibliography{ref}

\begin{thebibliography}{1}

\bibitem{Pramod-2021}
P.~N. Achar.
\newblock {\em Perverse sheaves and applications to representation theory}, volume 258 of {\em Mathematical Surveys and Monographs}.
\newblock American Mathematical Society, Providence, RI, 2021.

\bibitem{BBD}
A.~A. Be\u{\i}linson, J.~Bernstein, and P.~Deligne.
\newblock Faisceaux pervers.
\newblock In {\em Analysis and topology on singular spaces, {I} ({L}uminy, 1981)}, volume 100 of {\em Ast\'{e}risque}, pages 5--171. Soc. Math. France, Paris, 1982.

\bibitem{2004The}
P.~Caldero and M.~Reineke.
\newblock The bar automorphism in quantum groups and geometry of quiver representations.
\newblock {\em Annales- Institut Fourier}, 56(1):255--267, 2004.

\bibitem{deng2006frobenius}
B.~Deng and J.~Du.
\newblock Frobenius morphisms and representations of algebras.
\newblock {\em Transactions of the American Mathematical Society}, 358(8):3591--3622, 2006.

\bibitem{Fang2025}
J.~Fang, Y.~Lan, and Y.~Wu.
\newblock The parity of lusztig's restriction functor and green's formula for a quiver with automorphism.
\newblock {\em Algebras and Representation Theory}, 28(2):483--508, Apr 2025.

\bibitem{lusztig1990canonical}
G.~Lusztig.
\newblock Canonical bases arising from quantized enveloping algebras.
\newblock {\em Journal of the American Mathematical Society}, 3(2):447--498, 1990.

\bibitem{lusztig2010introduction}
G.~Lusztig.
\newblock {\em Introduction to quantum groups}, volume 110 of {\em Progress in Mathematics}.
\newblock Birkh\"{a}user Boston, Inc., Boston, MA, 1993.

\bibitem{lusztig1998canonical}
G.~Lusztig.
\newblock Canonical bases and hall algebras.
\newblock In {\em Representation theories and algebraic geometry}, pages 365--399. Springer, 1998.

\bibitem{OSNotes}
O.~Schiffmann.
\newblock Lectures on canonical and crystal bases of {H}all algebras.
\newblock In {\em Geometric methods in representation theory. {II}}, volume~24 of {\em S\'{e}min. Congr.}, pages 143--259. Soc. Math. France, Paris, 2012.

\end{thebibliography}
\end{document}